\documentclass[runningheads,a4paper]{llncs}

\usepackage{amssymb}
\usepackage{amsfonts,latexsym,graphicx} % new
\usepackage{graphicx}
\usepackage{url}
\usepackage{apacite}

\usepackage{color}

\usepackage{caption}
\newcommand{\keywords}[1]{\par\addvspace\baselineskip
\noindent\keywordname\enspace\ignorespaces#1}

\begin{document}

\emergencystretch 3em

\mainmatter  % start of an individual contribution

\title{Knots, Music and DNA}

\titlerunning{Knots, Music and DNA}

\author{Maria Mannone\thanks{I would like to thank the mathematician and musician Franck Jedrzejewski for his important suggestion to study the implications of knot theory for the mathematical description of musical gestures. I also would like to thank the mathematicians and musicians Emmanuel Amiot and Tom Collins, the composer Dimitri Papageorgiou, the mathematicians Anar Akhmedov and Ian Whitehead for the interesting feedback on the manuscript, and the mathematician Giuseppe Metere for the insightful discussions about 2-categories.}}
%
% if the names of the authors are too long for the running head, please use the format: AuthorA et al.
%\authorrunning{AuthorA and AuthorB (or AuthorA et al. if too long)}

\institute{University of Minnesota\\ \email{manno012@umn.edu}}


\maketitle

\begin{abstract}
\textit{Musical gestures connect the symbolic layer of the score to the physical layer of sound. I focus here on the mathematical theory of musical gestures, and I propose its generalization to include braids and knots. In this way, it is possible to extend the formalism to cover more case studies, especially regarding conducting gestures. Moreover,  recent developments involving comparisons and similarities between gestures of orchestral musicians can be contextualized in the frame of braided monoidal categories.
Because knots and braids can be applied to both music and biology (they apply to knotted proteins, for example), I end the article with a new musical rendition of DNA.
}
\end{abstract}

\section{Introduction}

The study of musical gestures is a growing field of research in performance \cite{leman, fortuna, global}, composition \cite{greek_1}, mathematics \cite{visi_miranda_2}, physics \cite{wand_timpani_1}, cognitive sciences \cite{lipscomb}, as well as signal processing and computer sciences \cite{greek_2, bergsland}, and the comparisons among different definitions of musical gestures \cite{cadoz}. The mathematical definition of musical gestures uses the formalism of category theory \cite{macLane}. A musical gesture is a mapping from a directed graph (skeleton) to a system of curves in space and time\footnote{Indicated in the literature as a topological space \cite{jmm}.} (body) \cite{jmm}; see Section \ref{App1}. Such a definition has been extended to conducting and to singing. In fact, singing can be described as the result of inner -- partially conscious and non-conscious -- movements \cite{global, phd_mannone}. Recently, criteria to compare gestures have been proposed, mainly using homotopy \cite{greek_2} and mathematically defining similarities among homotopically equivalent gestures producing similar effects on the final sound spectra \cite{gest_sim}.\footnote{Two gestural curves are called {\em homotopically equivalent} if there exists a continuous mapping that transforms one into the other; if one can be continuously deformed into the other.} Such a theory also allows comparisons between music and visual arts in light of the same gestural generators, highlighting similarities between gestures leading to music, and gestures leading to visuals, such as drawing and sculpting and even writing, as proposed in \citeA{cadoz}. 

Gesture studies can borrow the formalism of networks recently applied to the comparison of visuals \cite{leonidas}, where concepts from category theory such as colimit and limit are used \cite{macLane}. We can define networks of images as well as networks of sounds and connect them \cite{networks}. We can describe visuals as the result of some drawing gesture(s), and thus we can define similarity between music and visual arts \cite{gest_sim}. In this way, we can refine a technique to translate music into images, and vice versa -- tridimensional images into music \cite{music_image_book}.\footnote{With the mapping: height to pitch, depth to time and length to loudness.} We can argue that the most successful examples of translation between different art fields verify crossmodal correspondences \cite{spence, phd_mannone}. We may refer to several studies in the field of crossmodal correspondences in psychology, psychophysics, neuroscience and linguistics \cite{gestalt, nobile, gent5, zbikowski, zbi, roffler} to support this idea. Other applications of networks in music also refer to cognitive models \cite{popoff1}. Moreover, visuals and sound can often be seen as parts of some unique, unitary entity, the {\em audio-visual object} and the {\em theory of indispensable attributes} \cite{kubovy}. In Kubovy's studies, it is highlighted that the main function of the visual channel is detecting shape, while the main function of the auditory channel is localizing moving sound sources. According to Kubovy, the indispensable attribute for visuals is shape, and the indispensable attribute for sound is pitch. This may constitute a conceptual basis for a sonification of essential lines through pitch variations, one of the possible techniques.  

Following the input of recent studies \cite{jed, jed2, popoff2}, I try to advance the mathematical theory of musical gestures using concepts from knot theory \cite{Adam} and braided monoidal categories \cite{funct_knots}. Knots and braids can describe several musical situations, such as conducting gestures, and hand-crossing piano playing -- let us think of Webern's Variations Op. 27 No. 2 (1922).
The formalism of functors and n-category theory \cite{baez_cat} can be easily extended to embrace these applications \cite{funct_knots}. Moreover, knots and braids are successfully applied inside theoretical models in biological studies, such as DNA \cite{Austin}.
The reason is that the molecule of DNA, the Deoxyribonucleic Acid, is formed by ``pairs of molecular strands that are bonded together'' and that ``spiral around each other'' \cite[p. 181]{Adam}. Thus, the famous double helix can be mathematically modelled as a pair of curves mutually twisting, easily described as consecutive braids. Moreover, DNA molecules are then packed and knotted several times. The entire double helix can be turned, closed in circles and coiled; it can thus be studied as a topic of topology and knot theory \cite{top_DNA}. Knots can also be found in other molecules \cite{bio1}.
Among the connections of DNA, knot theory and braids, we may consider Lissajous knots \cite{liss, lissajous2}, but first of all knotted proteins \cite{bio2, bio3}. Musicians may already be familiar with Lissajous for his experiments with tuning forks \cite{tuning_forks}, and for the patterns obtained with Chladni plates \cite{Chladni}.

Joining the concept of musical and visual gestures with knots and biology, I end the article with the description of an original musical piece whose structure is derived from the shape of DNA. DNA has already been the topic of musical transpositions and algorithm developments \cite{DNA_former2}.\footnote{See also the webpage \url{https://www.yourdnasong.com.}} Here the approach is different, but it can be interpreted as complementary. I use the same gestural generator for both the visual shape of DNA and the musical patterns, as I will describe in Section \ref{DNA_Section}. In fact, the mathematical theory of musical gestures has both a descriptive role (for analysis), and a prescriptive role (for composition of new works); see \citeA{origin} for an example of a creative technique. In fact, these examples show how inputs from theory and analytical approaches can be employed to enhance musical creativity. I will discuss this more fully in the conclusion. Sections \ref{App1} and \ref{App2} concern the mathematical definition of gestures, hypergestures and some basic definitions of 2-category theory. Section \ref{why} introduces knots in the framework of the mathematical theory of musical gestures. Section \ref{quantum} deals with monoidal categories and references to quantum mechanics. Section \ref{network_par} presents a generalized approach to networks using tools from category theory. Section \ref{DNA_Section} analyses an original musical rendition of DNA structure. The Conclusion summarizes the results and discusses further research. 

\section{Mathematical definitions}

\subsection{Gestures and Hypergestures}\label{App1}

The mathematical theory of gestures has been developed to explain the embodiment of music in performance. In fact, in order to produce sounds, musicians have to interact physically with instruments, transforming the symbolic indications contained in the scores, in a sequence of precise movements. The first definition is given in \citeA{jmm}.
Fig. \ref{gesture} shows the most simple case: a mapping from a skeleton constituted by a directed graph with only two points and an arrow connecting them, to a body given by a curve connecting two points in a topological space. This is completely general, and thus the scheme can be applied to any musical (but not only musical) gesture.
The skeleton is indicated with capital Greek letters, and the curves in the topological space $X$ with the symbol $\vec X$. Thus, a gesture can be notated as $Hom(\Delta,\vec X)$. In \citeA{jmm}, \citeA{global} and \citeA{gest_sim} it is used as the non-standard notation $\Delta@\vec X$. 

In a nutshell, the simplest gesture is a path connecting two points in the space. The entire gesture can be seen as a single point in a space of higher dimensions. A hypergesture is a path connecting two points, but these points are two gestures. Each point in the hypergestural curve is a gesture, so, in principle, it is possible to transform a gesture into another via an infinite collection of intermediate gestures. As a more practical example, we might think of a pianist making a movement to play a note. When the pianist moves the arm to play notes in several points of the keyboard, he or she is connecting gestures, and thus making a hypergesture.
% When we use the tensor product, we mean to give an idea of the dimension change.

The entire gesture can be considered as a point in the space of hypergestures; a curve connecting two points in such a space is a hypergesture.
\begin{figure}
\centerline{
\includegraphics[width=7cm]{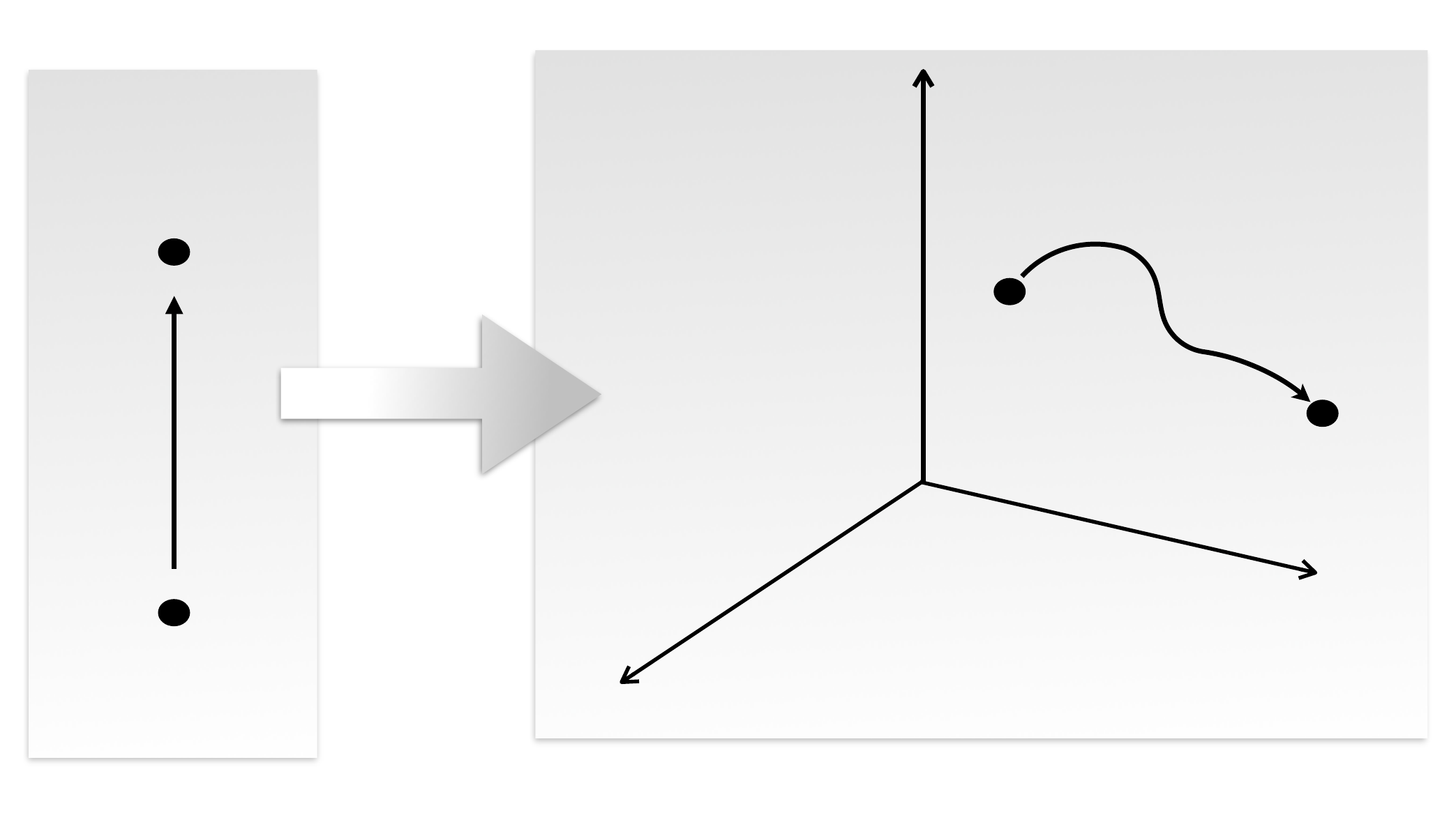}}
\caption{\label{gesture} \footnotesize A gesture is a mapping from a directed graph (left) to a system of continuous curves in a topological space (right). The image shows the simplest case, where we have only one arrow, and then only one curve.}
\end{figure}
In general, skeleta and bodies have more arrows. Recently, skeleta also involving branching have been studied \cite{global}, and extensions of the gesture theory to singing and conducting have been developed \cite{ToM_III_voice, ToM_III_conducting}. In singing, there are inner movements, only partially controlled by the performer. External movements, such as posture changes and hand movements/face expressions, as well as visual metaphors, help the singers to shape their vocal system correctly in order to pick the desired pitch and timbre.
Comparisons between gestures of different orchestral performers, as well as between conductor and orchestra, are described in \citeA{gest_sim}, where the formalism of 2-categories is involved.

\subsection{2-Categories}\label{App2}

Hypergestures can be more naturally described by means of 2-categories \cite{gest_sim}. In a nutshell, a 2-category is a category with morphisms between morphisms \cite{baez_cat}. More precisely, it is possible to talk about {\em equivalence classes of hypergestures} to have a 2-category; see Theorem \ref{quotient} for more details.

Fig. \ref{face} shows a smiling simplified face, Fig. \ref{2-cat} shows a 2-category and Fig. \ref{2-2-cat} shows a nested structure with a 2-category whose objects are other 2-categories.
In Fig. \ref{2-cat}, there are two objects, with two different morphisms connecting them. The two objects are two categories, and the morphisms are functors. The double arrow transforming one functor into the other is a natural transformation.
In principle, we can have n-categories and infinite-categories \cite{infinite}. 2- and n-categories are also important for theoretical physics \cite{baez_cat}.

Let us now examine musical examples of 2-categorical thinking. One can think of two different ways to perform the same musical sequence, one with loudness {\em piano}, the other with {\em forte}. One can transform the first version into the second one, via a {\em crescendo}. However, there are different ways to make a {\em crescendo}: slower or faster, for example. One can define a {\em natural transformation} ``tempo'' to transform the slower crescendo into the faster crescendo. A ``tempo'' transformation would thus be an arrow between arrows. A fast crescendo can be played by strings, or by winds; the same is true for a slow crescendo. There can also be transitions from strings to winds in orchestration, and one can describe this via another arrow, and so on. All these musical structures can be easily described using the formalism of 2- and n-categories.
\begin{figure}
\centerline{
\includegraphics[width=4cm]{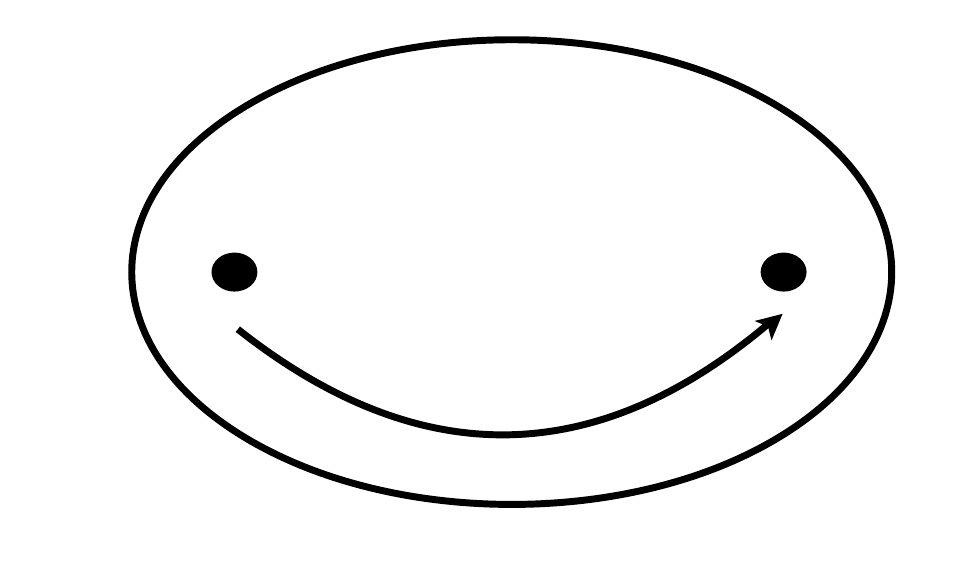}}
\caption{\label{face} \footnotesize A smiling face; a humour depiction of a morphism connecting two objects.}
\end{figure}
\begin{figure}\label{2-cell}
\centerline{
\includegraphics[width=4cm]{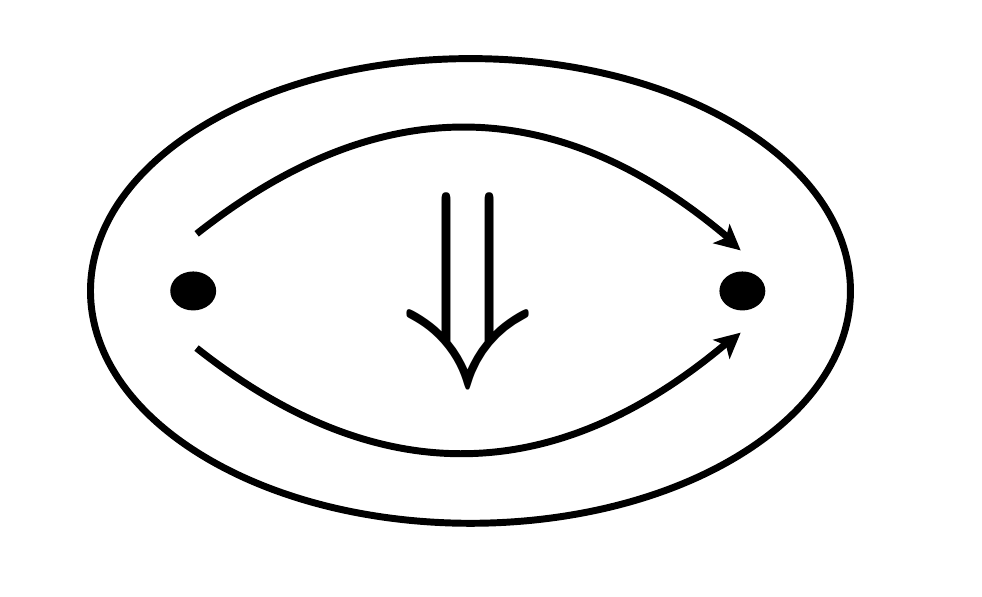}}
\caption{\label{2-cat} \footnotesize An example of a 2-category.}
\end{figure}
\begin{figure}
\centerline{
\includegraphics[width=7cm]{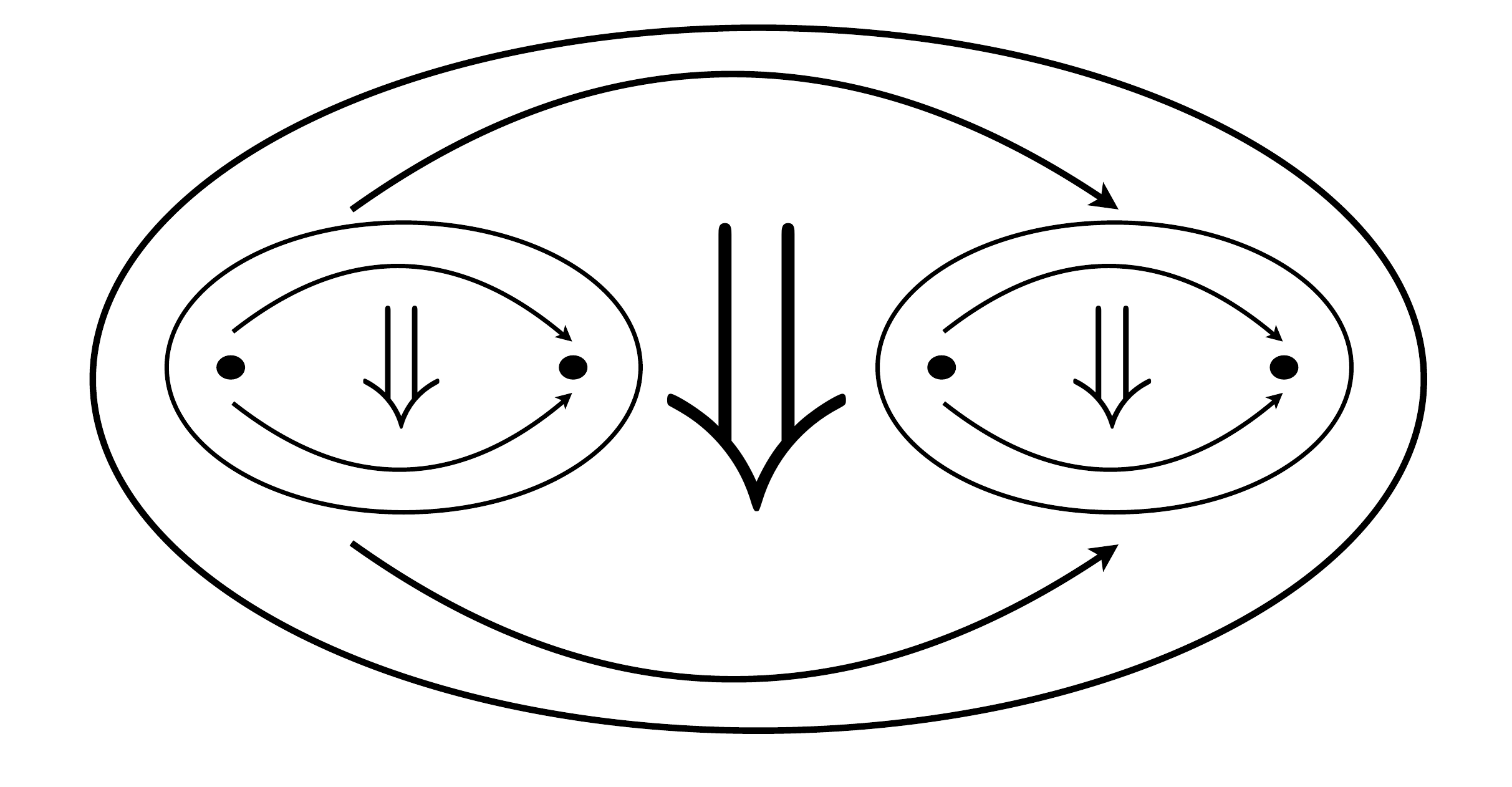}}
\caption{\label{2-2-cat} \footnotesize A 2-category containing two 2-categories, giving a 4-category.}
\end{figure}

\section{Why (k)not?}\label{why}

%As cited in the Introduction, a musical gesture is defined as a mapping from a directed graph (skeleton) into a system of continuous curves in a topological space (body) \cite{jmm}. In more physical terms, we can talk about space and time. We can have gestures, creating a gesture of gestures, called {\em hypergesture}, see details in Appendix \ref{App1}. This definition has been generalized to branched skeleta \cite{global}, and applied to several musical cases, comprised of the two ``extreme'' cases: conducting---gesture without direct sound production, and singing---inner gesture and inner musical instrument \cite[Chapter 37]{ToM_III}, \cite{phd_mannone}.

Jedrzejewski \cite{jed2} raised a question concerning the absence of any formalism, in Mazzola's mathematical theory of gestures, to describe gestures that inherit some characteristics from other gestures. The problem of gestural inheritance had already been highlighted years before the development of such a mathematical theory \cite{cadoz}. 
A way to overcome this difficulty may be the use of monoidal tensor categories \cite{funct_knots} and of knot theory \cite{Adam}. Due to its nature, a musical gesture -- let us say, the objectivized trajectory of a point, for example the tip of the conducting baton -- may contain repetitions of the same points in the path. This fact, as pointed out by Jedrzejewski, can be easily framed in knot theory. One can imagine a gesture as a closed path, and we can investigate whether a gestural curve may be reduced to the unknot, or ... not. A very first and intuitive example of this is the conducting gesture of the right hand, closed and cyclic.
Knot theory is related to braids, and braided (monoidal) categories \cite{baez_cat} are a powerful tool of analysis that can be used in gesture theory \cite{jed2, jed3}. Another comparison between knot theory and gesture theory may come from {\em billiard knots}, i.e., the study of the trajectory of a ball on a billiard table. It is possible to study the correspondence between the straight lines reflected in the wall, and a curve bending on itself described in knot theory \cite{lissajous2}. One might compare the billiard ball's trajectory with the skeleton of the gesture, and the smooth knot path with the body of the gesture. There is not only one knot corresponding to a billiard ball's path, just as there is not only one body (of a gesture) corresponding to a skeleton (of a gesture).

Moreover, one can extend the comparison between knot theory and gesture theory through the concept of {\em link}: for example, the trajectories of the two hands of an orchestral conductor do not intersect, but may be linked. This would constitute a development of analysis of conducting gestures in mathematical terms.

In summary, a description of musical gestures closer to their physical reality and their incredible complexity joins several approaches so far envisaged: branching and global skeleta \cite{global}, braided monoidal categories \cite{baez_cat, jed2}, knots \cite{Adam, jed}, 2-and n-categories \cite{gest_sim, infinite}, and networks \cite{networks}.\footnote{The connection between different theories and approaches can go on and on. For example, it is possible to frame the world-sheets \cite{global} in Cobordism of complex spaces $(Cob_C)$.}
For example, the skeleton of a gesture can be branched \cite{global}, but also twisted in a braid, and the body will consequently be twisted; see Fig. \ref{fig_1}.
Also, we can have knots when the skeleton is ``simple''; see Fig. \ref{fig_2}.
Analysing conducting gestures, it seems that they are (all?) realizations of the unknot; see Fig. \ref{fig_3}.

In principle, we can also use concepts from physics, such as QFT (Quantum Field Theory), fields (as defined in physics), and path integrals for each couple of points, to describe the possible realizations of a skeleton once the initial and final points (within the topological space) have been chosen -- in all my examples, I chose the subspace topology of knots.

 \begin{figure}
 \centerline{
\includegraphics[width=5cm]{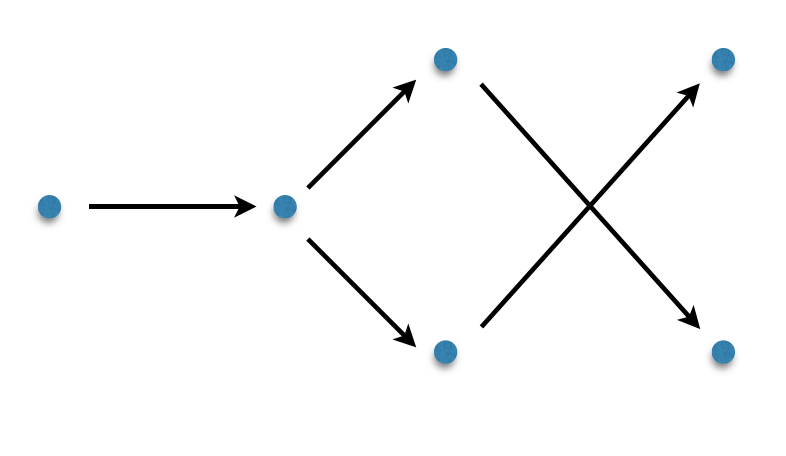}}
\caption{\label{fig_1} \footnotesize A skeleton that is first branched and then braided.}
\end{figure}

 \begin{figure}
 \centerline{
\includegraphics[width=7cm]{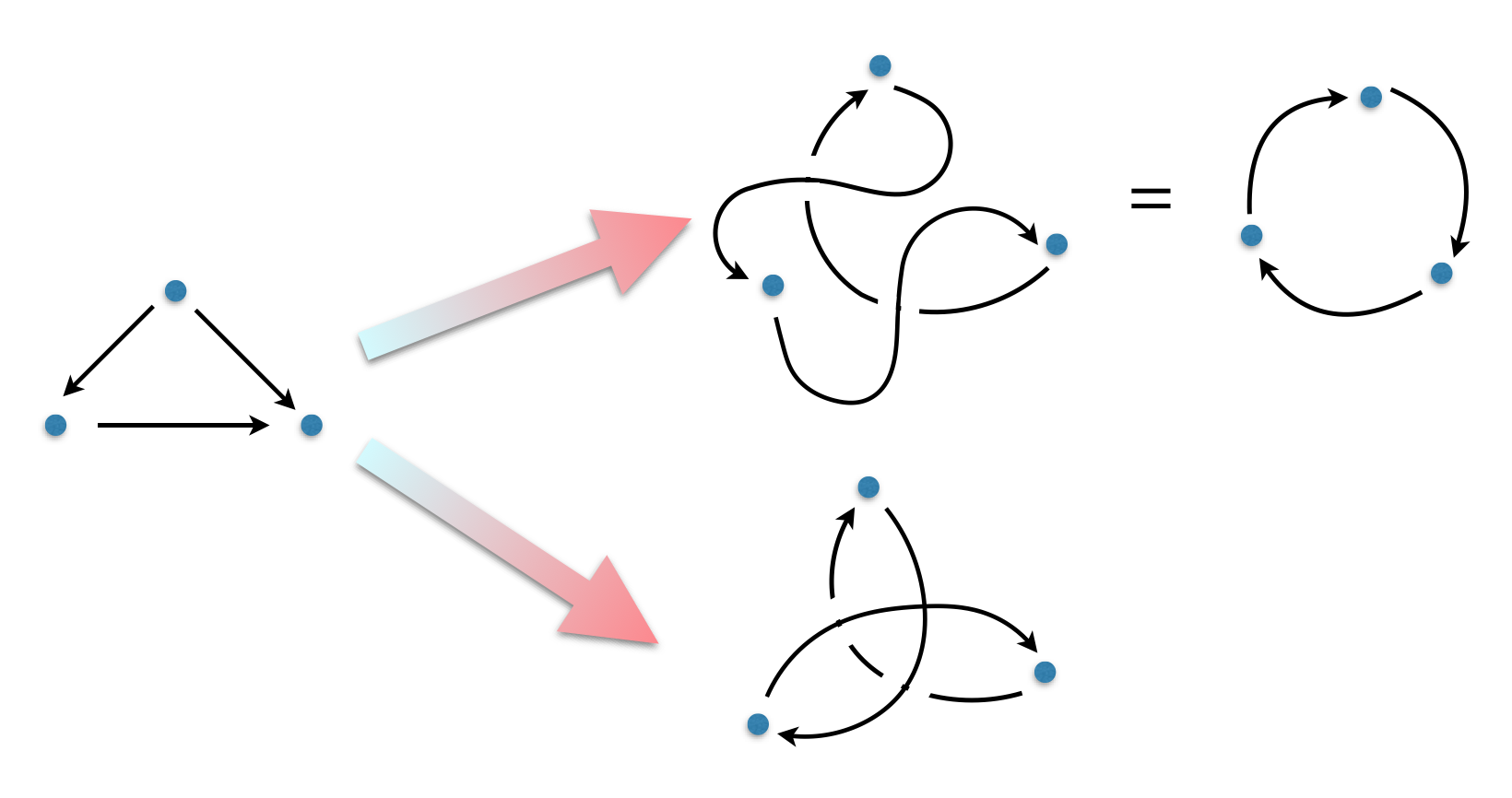}}
\caption{\label{fig_2} \footnotesize The three-point skeleton can be mapped either in the unknot, or in a trefoil knot (if we see the body of the gesture as a closed line).}
\end{figure}

 \begin{figure}
 \centerline{
\includegraphics[width=8cm]{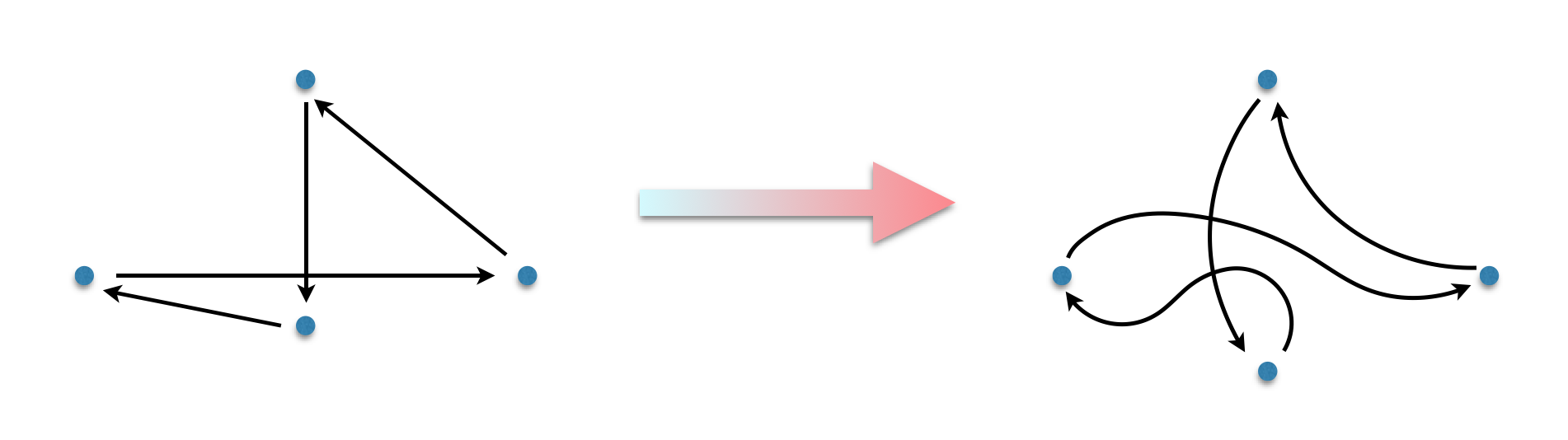}}
\caption{\label{fig_3} \footnotesize The skeleton of a quaternary conducting movement, and its realization as a system of continuous curves. If it is seen as a unique closed curve of knot theory, conducting gestures correspond to the unknot.}
\end{figure}

One can imagine developing a ribbon category into a conducting gesture with modulo 3; see Fig. \ref{fig_4}.
The definition itself of gesture can be extended and improved in knot frameworks \cite{jed2}.
% In this work, it is discussed the important topic of gestures inherited one from the other, an idea absent from the first studies on the field \cite{jmm}.
Returning to the first mathematical definition of a gesture \cite{jmm}, and to the use of 2- and n-categories \cite{baez_cat, infinite} applied to music \cite{gest_sim}, one can further define the concept of hypergesture.

 \begin{figure}
 \centerline{
\includegraphics[width=8cm]{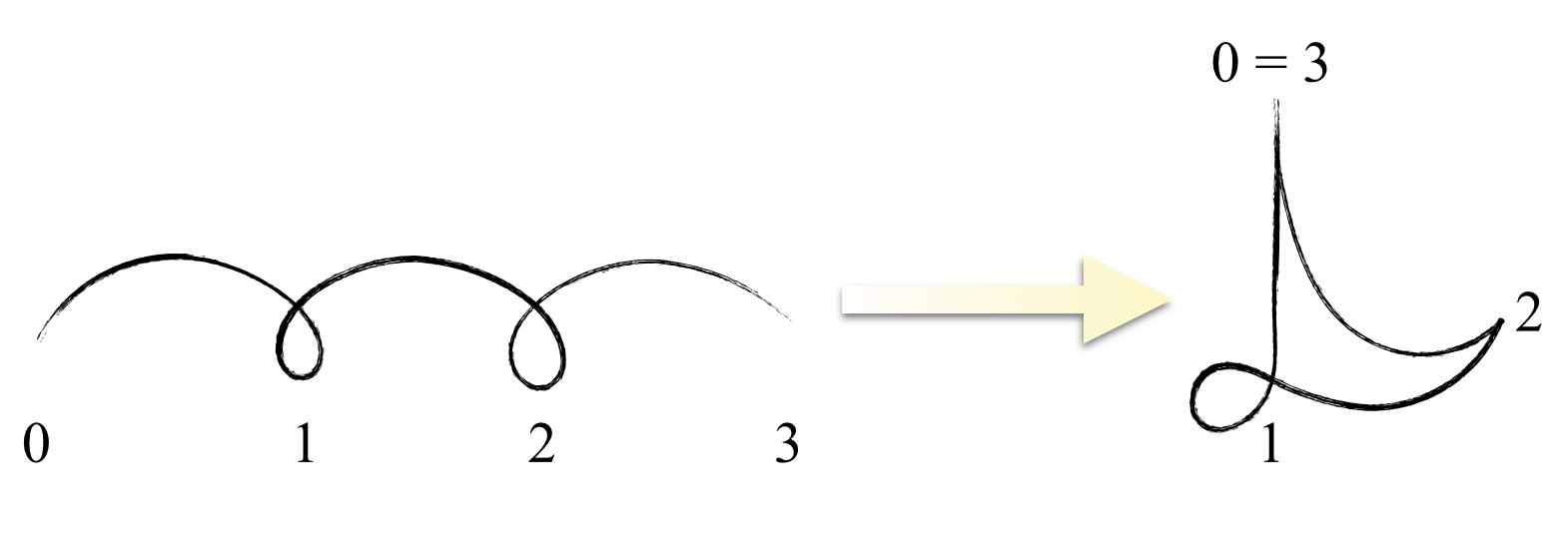}}
\caption{\label{fig_4} \footnotesize The ribbon-category useful to describe the graph on the left can be closed, giving a ternary-time conducting gesture, modulo 3.}
\end{figure}

\section{Hypergestures, Monoidal Categories and possible openings to Quantum Mechanics and Music}\label{quantum}

\subsection{Hypergestures}

%In \cite{jed2} are introduced monoidal categories to describe musical gestures.
%Here, we refer of some simple concepts of tensor products and braided categories.

%A hypergesture is a gesture of gestures, as previously defined \cite{jmm}.

In the space of hypergestures, points are gestures.
Let $\vec X$ be a space of gestures, let $\Delta, \Gamma$ be two skeleta, and let $g_1$ and $g_2$ be two gestures.
In the space $\vec Y$, the hypergesture $l_1$ connects point $g_1$ with point $g_2$, where
\begin{equation}
g_1=Hom (\Delta, \vec X),\,\, g_2=Hom (\Gamma, \vec X).
\end{equation}
But gestures, and consequently also hypergestures, are mappings from a skeleton to a system of curves in a topological space: so, given a skeleton $\Xi$, there is $h_1$ that maps $\Xi$ into $\vec Y$. But the curve in Y is $l_1$ connecting $g_1$ and $g_2$, if the skeleton $\Xi$ is an arrow connecting two points.
One can thus describe the hypergesture, $h_1=Hom(\Xi,\vec Y)$, as a mapping from $\Xi$ to $l_1=Hom(g_1,g_2)$, i.e., $h_1=Hom(\Xi,(Hom(g_1,g_2)))$.
If there are two possible paths, $h_1^{\alpha}$ and $h_1^{\beta}$ from $\Xi$ to $\vec Y$, one can define a morphism of morphisms $\alpha\beta$ connecting them. The same argument can be applied to the gestures from $\Delta$ to $\vec X$ and from $\Gamma$ to $\vec X$, as well as to the hypergestural curve(s) in $\vec Y$ space. Using the formalism of 2-categories, we obtain the diagram shown in Fig. \ref{alpha_1}.
 \begin{figure}
 \centerline{
\includegraphics[width=7cm]{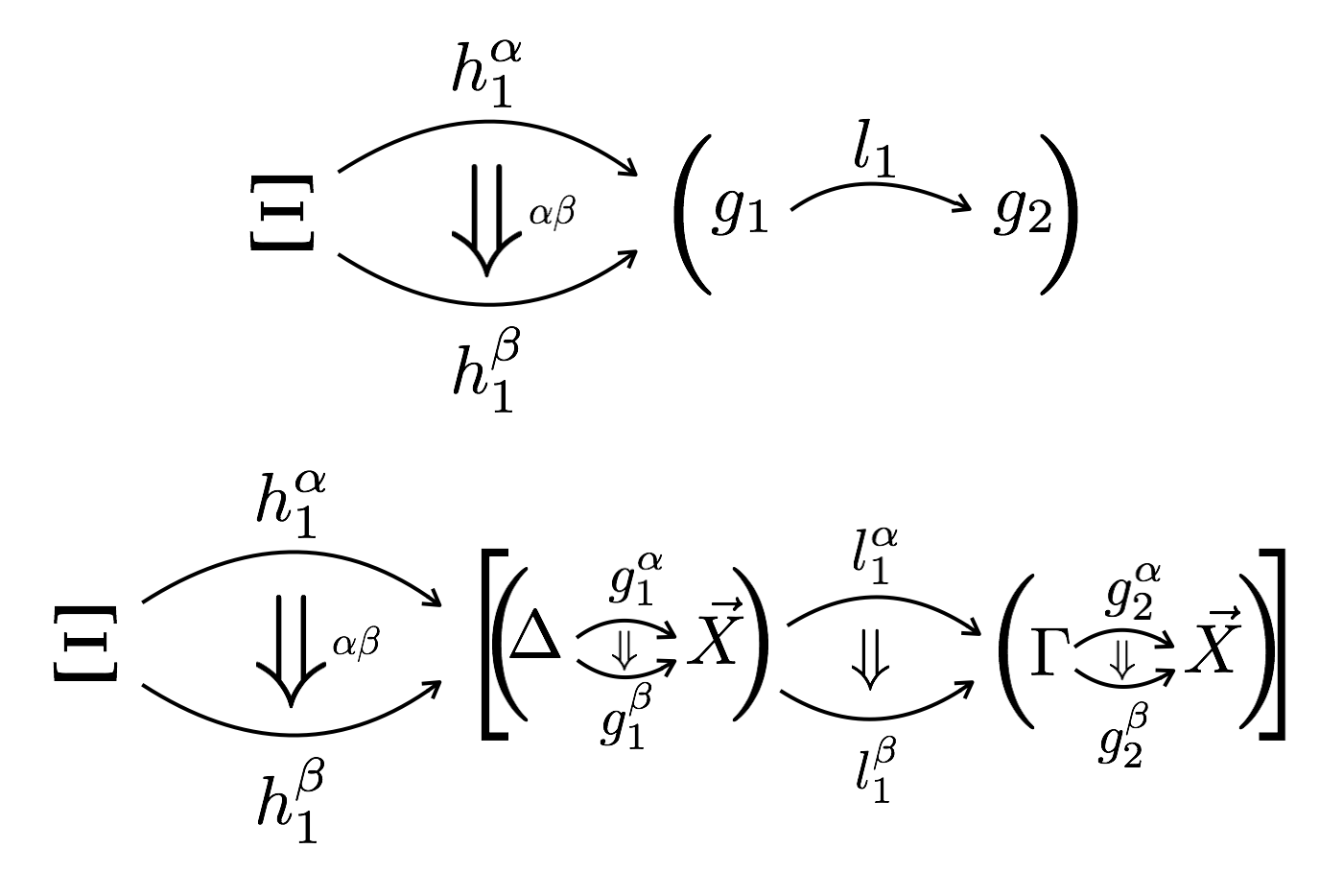}}
\caption{\label{alpha_1} \footnotesize For a given skeleton $\Xi$ corresponding to the simple structure point-arrow-point, we sketch here a 2- (and 4-) categorical description of morphisms of morphisms. We have a recursive structure of arrows and arrows between arrows.}
\end{figure}
Following the intuitive concept of nested gestures given by $Hom(Hom(...))$, one may argue for the existence of {\em gestural modulations}, the gestural analogues of frequency modulations.

One can summarize this last result in a theorem:
\begin{theorem} Gestures' and hypergestures' structures can be defined recursively.  \end{theorem}
\begin{proof}
\footnotesize Let $\Delta, \Gamma, \Xi$ be skeleta, and $\vec X$ a topological space. Gestures $g_1,\,g_2$ are defined as $g_1=Hom(\Delta,\vec X)$ and $g_2=Hom(\Gamma,\vec X)$. Let $\Delta = \Gamma = \Xi =$ point-arrow-point. Thus, both $g_1$ and $g_2$ are lines connecting two points. A line in $\vec Y$, the space of hypergestures, connects $g_1$ and $g_2$; see Fig. \ref{diagram}. If we generalize such a structure to N-dimension, we will still have two points connected. A N-hypergesture with point-arrow-point skeleton connects (N-1)-hypergestures, each (N-1)-hypergesture with point-arrow-point skeleton connects (N-2)-hypergestures, and so on. Finally, the N-structure can be reduced to a collection of pairs of points in the $X$ space ($\vec X$ being the space of curves with points in $X$).
\end{proof}

 \begin{figure}
 \centerline{
\includegraphics[width=7cm]{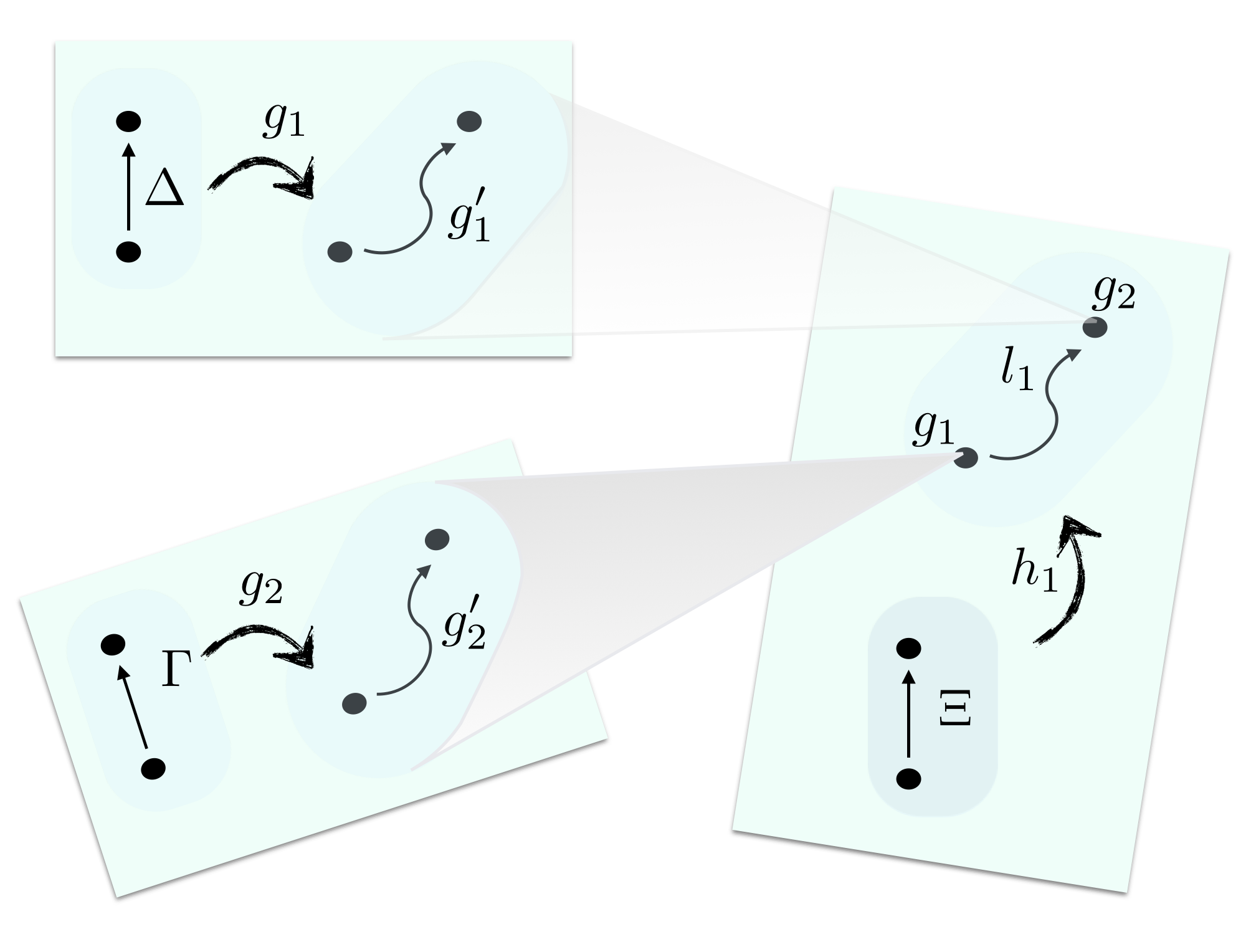}}
\caption{\label{diagram} \footnotesize Recursive structure of gesture and hypergesture structure. There is a slight abuse of notation here, because with $g_1,\,g_2$ we indicate the mapping from skeleta to the body of gestures but also the body of gestures themselves. For this reason, we indicate the latter with $g'_1,\,g'_2$. Let $p_1,\,p_2$ be two points connected by the first gesture, and let $p_3,p_4$ be two points connected by the second gesture. We have: $g'_1=Hom(p_1,p_2)$ and $g'_2=Hom(p_3,p_4)$. Thus, $l_1=Hom(g'_1,g'_2)=Hom(Hom(p_1,p_2), Hom(p_3,p_4))$. A higher-order hypergesture can be written as $Hom(Hom(..., Hom(p_1,p_2),Hom(p_3,p_4)))$. }
\end{figure}

Exploring modulation in the sense of a gradual change, one can also refer to tempo modulation, tonal modulation and intensity modulation. In this way, it is possible to develop a categorical approach to musical concepts with generalized functors and natural transformations.\footnote{In initial literature about mathematical gesture theory \cite{jmm} contravariant functors are often used.}

To determine if it is possible to apply such a recursion to the 2-category structure, one can see more details about the associativity of hypergestures using their quotient space.\footnote{The quotient set $X/\sim$ is the set of all possible equivalence classes in $X$. If $X$ is a topological space, then $X/\sim$ is the quotient space.} 

\begin{theorem}\label{quotient}
The composition of hypergestures (paths) is associative up to a path of paths.
\end{theorem}
\begin{proof}\footnotesize
To define 2-categories, we need to prove that the associativity property is satisfied. Let $X$ be a topological space, let $\vec X$ (or $X^I$, $I\rightarrow X$) be the space of paths within $X$, with $I=[0,1]$ the unitary interval. Let $f,\,g$ be two paths that connect points, and let each point be a gesture. The composition of $f,\,g$ needs a reparametrization: $f(x)$ if $0<x<1/2$, and $g(x)$ if $1/2<x<1$. While composing it with another path $h$, $h(gf)=(hg)f$, the parametrization is not the same, and we no longer have a category. However, the problem seems to be solved if we define the reparametrization $r(t,x)$ as $r(0,x)=h(gf)(x)$ and $r(1,x)=(hg)(f(x))$. The composition of two paths is associative up to a path of paths. If we try to compose two paths of paths, we find that it is associative up to a path of paths of paths.  
We can use, instead of paths, equivalence classes of paths. In fact, the multiplication of equivalence classes is associative.
Provided that objects (points) are gestures, they are called 0-cells. Hypergestures (paths between points) are 1-cells, and hyper-hypergestures (paths of paths) are 2-cells. We conclude that the composition of 1-cells is associative. If we are looking at paths and 2-cells (and not higher-order combinations), we can redefine 2-cells, taking equivalence classes up to a 3-cell.

% Provided that a 1-cell is a gesture (an arrow) and a 2-cell is a hypergesture (arrow between two arrows, see Figure \ref{2-cell}), a 3-cell is a connection of two hypergestures, that is, a path of paths. Thus, 2-cells compose up to a 3-cell.

% We have two possible scenarios: 1. either we consider associativity in N-cells up to a (N+1)-cell, or 2. we can consider not the paths, but the equivalence classes of paths {\em up to a 3-cell}, if we are only interested to 2-cells and not to higher-order cells.
\end{proof}
In summary, we can redefine hypergestures as {\em equivalence classes of hypergestures} to formally have a 2-category. The alternative solution is not to limit our analysis to 2-categories but to look instead at N-categories, being aware that the associativity is verified {\em up to a (N+1)-cell}, with $N\rightarrow\infty$.

% Here and in the following, while talking about {\em hypergestures}, we will implicitly assume {\em equivalence classes} of hypergestures. 

\subsection{Tensor products}
Let us consider a gesture $g_1$, represented by a curve in $\vec X$. In the space $\vec Y$ of hypergestures, $g_1$ is a point. To indicate the embedding in the higher dimensional space, we can introduce tensor products here, choosing the notation $g_1\in\vec X$, and $g_1\otimes 1\in\vec Y$. The same argument can be applied to another gesture; let us say $g_2$: $g_2\in\vec X$,
$g_2\otimes 1\in\vec Y$. Curves $g_1$ and $g_2$ are connected by a surface in $\vec X$, that is a line in $\vec Y$; let us call it $l_1$.\footnote{The described surface is defined as a world-sheet in \citeA{mcm15} and \citeA{global}.} Thus, $l_1\in\vec Y$. $l_1$ is the path that connects $g_1$ and $g_2$: $l_1=path(g_1,g_2)\in\vec Y$. We have the following theorem:
\begin{theorem} Let $g_1$ and $g_2$ be two gestures in $\vec X$. The hypergesture connecting them, i.e., the path $l_1$ connecting them, lives in 1-dimensional higher space $\vec Y$, and thus the tensor product of $g_1$ and $g_2$ lives in $\vec Y$ too:
\begin{equation}g_1\otimes_{l_1}g_2\in\vec Y.\end{equation}
\end{theorem}

\begin{proof}
\footnotesize
The gesture $g_1$ is represented by a curve in $\vec X$, but it is a point in $\vec Y$. Thus, it can be considered as a trivial curve in $\vec Y$, and we write this as $g_1\otimes 1\in\vec Y$. Similarly, $g_2\otimes 1\in\vec Y$. The path $l_1$ connecting $g_1,\,g_2$ is a curve in $\vec Y$, and thus $l_1\in\vec Y$. If we connect the two trivial curves $g_1\otimes 1$ and $g_2\otimes 1$, we get a curve in $\vec Y$: 
\begin{equation}
[(g_1\otimes 1)\rightarrow(g_2\otimes 1)\in \vec Y]\Rightarrow[(g_1\rightarrow g_2)\otimes 1]\in\vec Y.
\end{equation}
Because $l_1=path(g_1,g_2)$, we have $l_1=path(g_1,g_2)=(g_1\rightarrow g_2)\otimes 1$. If we re-write 
$(g_1\rightarrow g_2)\otimes_{l_1}1$ (where the label $l_1$ indicates the specific choice of path) as $g_1\otimes_{l_1} g_2$, we have $g_1\otimes_{l_1} g_2\in\vec Y$. As an alternative proof, we may consider the collection of (infinite) points from $g_1$ to $g_2$ as a collection of trivial curves in $\vec Y$: $\left\{g_1,g_{1n},g_{1(n+1)},...,g_{2}\right\}\otimes 1$. They constitute the path $l_1\in\vec Y$, and they are already in $\vec Y$ by definition.
\end{proof}

This opening toward tensor products helps the introduction of braided categories. When one starts with a couple of gestures $g_1$ and $g_2$, and a path $p_1$ brings the first into the second, and a second path $p_2$ does the opposite, one can describe the situation as $g_1\otimes g_2\rightarrow g_2\otimes g_1$, with the two paths being $p_1:g_1\rightarrow g_2$ and $p_2:g_2\rightarrow g_1$, that compose the braid and are in $\vec Y$.

In summary, in $\vec X$ the gestures $g_1$ and $g_2$ are curves (1-dimension), and the connecting path is a surface (2-dimension) embedded in a tridimensional space in a first approximation;
in $\vec Y$, $g_1$ and $g_2$ are points (0-dimension), and their connecting path is a curve (1-dimension). Moving toward higher levels in the hierarchy, the dimension decreases. In music, the space $\vec X$ can be the space of hands' movements, and $\vec Y$, the space of arms' movements, for example.
Finally, the topic of cohomology discussed in \citeA{global} can be reviewed in this context of tensor (monoidal) categories. 

\subsection{Density Matrices for Gestures?}

Because quantum mechanics can also be described in terms of monoidal categories \cite{baez_cat}, especially in reference to qubits \cite{qubit}, one may describe the transition from some gestural states (as points in the spacetime) to other gestural states in terms of probability matrices. For example, some gestural curves may be more suitable than others, and then their probability is expected to be higher. Matrices would be useful in gestural similarity comparisons, both quantifying the gestures themselves, as well as their final spectral results (as separated matrices).

The general idea is that a matrix can characterize a gestural configuration, or the probability transition from a set of parameters to another. The matrix formalism can simplify the comparisons between gestures of different artistic codes.
In the literature, there are matrices to describe transitions between homometric states \cite{amiot}.

Let us consider, as a minimal gesture, two points connected by an arrow. One can simply think of the two points as the two energy levels of a 2-level atom, and the arrow between them, as a process in the atom. A probability matrix would help to describe the likelihood of such a process to happen -- for example, with the transition from up to down spin state.

Criteria to quantify memory in quantum systems have been defined \cite{mannone_physics}. Thus, following the analogies between quantum mechanical matrices and musical matrices, criteria to quantify gestural similarity can be developed. There have been attempts to adapt the formalism of non-Markovianity (the amount of memory in quantum systems) to music \cite{non_markov}; such a gestural extension would be an ideal development and improvement of that former work.
 
Another reference to quantum mechanics can be given by comparisons of 2-level atoms with the generalized intervals of Lewin's theory \cite{lewin}, with the general definition of distance between an element and another. One more reference is the concept itself of interval in physics, involving space and time: it can be compared with a multi-parameter comprehensive notation of a musical interval.

Another connection is given by Dirac notation applied to networks, as described in \citeA{kevin} -- where the problem of memory and neural networks is also discussed. Category theory has been applied to simple qubit states \cite{cat_physics}, and one can easily envisage further research developments involving category theory and network theory applied to music, using Dirac notation for musical generalized intervals.  I will deal with networks in Section \ref{network_par}. Finally, further developments of the mathematical theory of musical gestures should also involve the relations between hand movements, magnetic field variation and sound production in theremin playing. The theremin is already an object of scientific research \cite{theremin}, and such a new study would connect once more physics, category theory and art. 

 \section{Functorial approach to Networks}\label{network_par}

The use of braids and knots can be extended to networks of musical gestures. Networks and categories for music are already a topic of research, with the transformation of K-nets within a categorical framework \cite{popoff1}.
Once a musical network is defined, one can define another network, for example a visual one \cite{networks}. A network of images contains a collection of images, one transformed into another, as proposed by \citeA{leonidas} and \citeA{tulsiani}. Musical gestures can be compared with gestures that generate visual artworks via gestural similarity techniques \cite{gest_sim}. What can be done for single, isolated gestures, can be transferred to structures of gestures such as networks \cite{networks}. For this reason, one can connect musical networks with visual networks via functors. In the formalism of 2-categories, one can transform a functor $F_1$ into a functor $F_2$ via a natural transformation. One is thus defining networks of gestures, and hypergestures of networks. Fig. \ref{alpha_2} summarizes these concepts. 
\begin{figure}
\centerline{
\includegraphics[width=7cm]{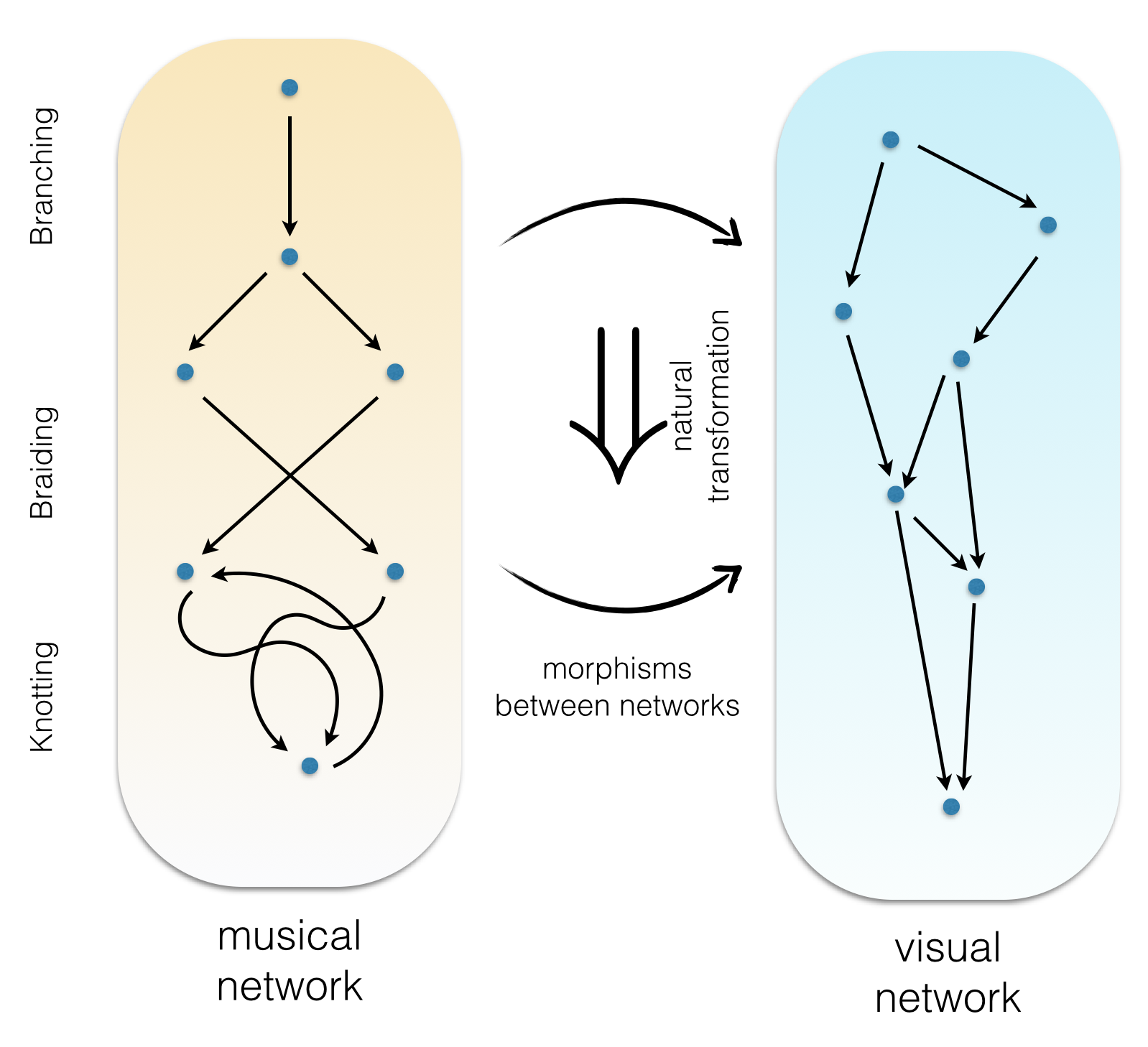}}
\caption{\label{alpha_2} \footnotesize A network mapped onto another. One can map musical networks onto visuals, and vice versa. Here, the structure of each of them is arbitrary. The network on the left shows branching, braiding and knots.}
\end{figure}

One can connect objects inside a network via morphisms, but one can also connect networks to other networks and transform morphisms of a network into morphisms of another network.\footnote{One can define functors between networks, that map objects of a network onto objects of another network, and arrows of a network onto arrows of the other network, and preserve the composition of morphisms and identity morphisms.} Fig. \ref{alpha_3} shows the transformations $\alpha,\,\beta,\,\gamma$ between morphisms in two different networks. This is another way to approach the problem of gestural similarity, inside and between networks.
Thus, one can also have knots, branching and braiding in networks: some structures that one finds in gestures and hypergestures may be replicated on a larger scale, as happens with fractals. Here, I will not deal with fractals, but with other interdisciplinary applications.

The scope of this section is not concerned with unfolding details, but with envisaging some guidelines to create a general, abstract model, one that may help modeling complex musical systems using simple ideas.

In summary, I have shown how to generalize the mathematical theory of gestures, including elements of knot theory and monoidal braided categories. One can find these elements inside musical networks. One can connect musical networks with networks of visual arts, for example, and the connections between these networks can be investigated via category theory. Future developments may also involve gestural K-nets, and all combinations of these elements. These ideas may be useful not only for analytical purposes, but also for creative ones. In the following section, I discuss some possible ways to translate into music the structure of DNA, topoisomerase and knotted proteins, starting from a simplification of their shape, and of course using their mathematics.

\begin{figure}
\centerline{
\includegraphics[width=5cm]{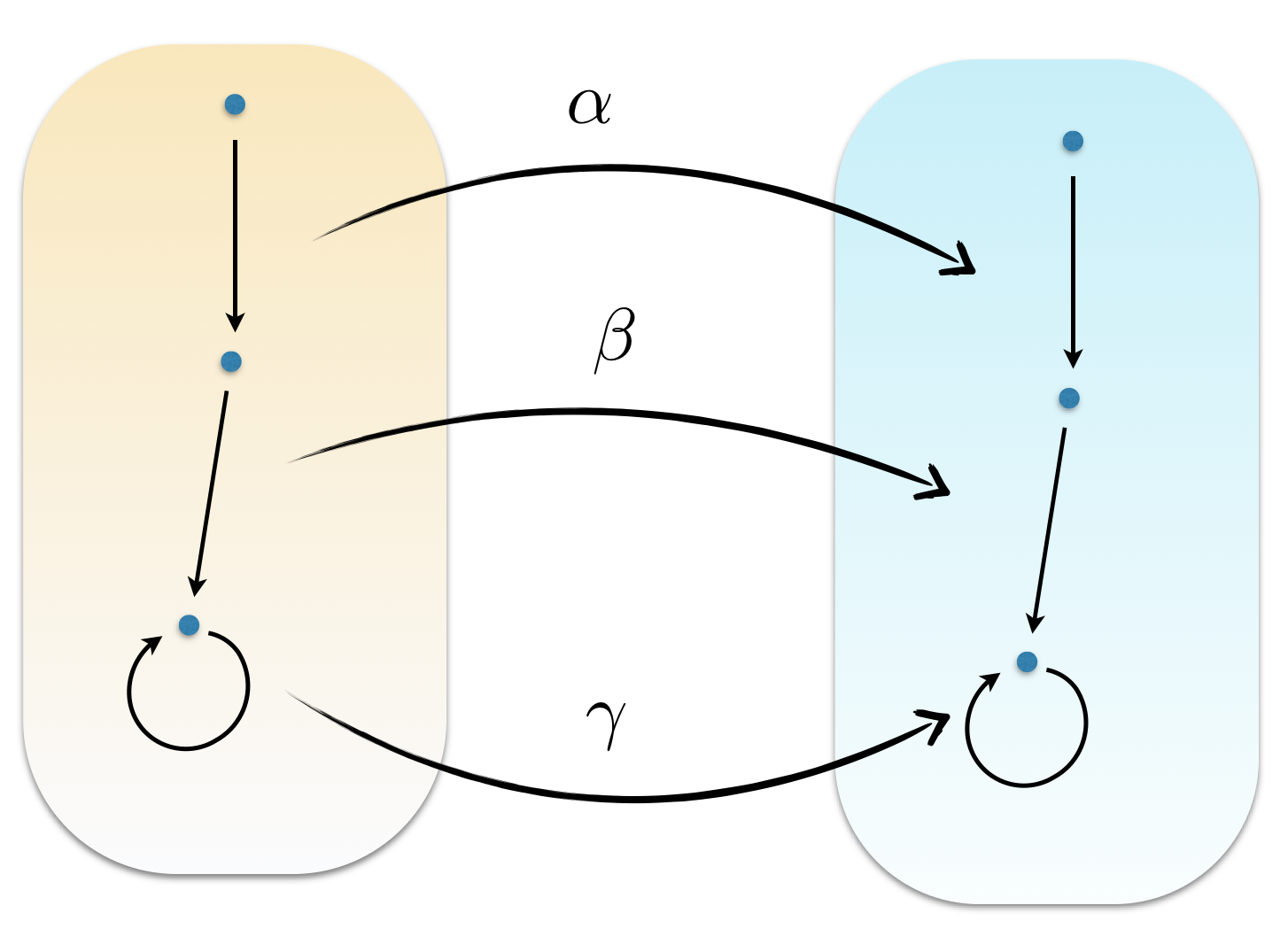}}
\caption{\label{alpha_3} Morphisms between the arrows of two separated networks. I can also define morphisms from $\alpha$ to $\beta$, from $\beta$ to $\gamma$, and their composition.}
\end{figure}

\section{From DNA to Music}\label{DNA_Section}

Topics in biology such as DNA and knotted proteins are also the object of mathematical studies \cite{top_DNA, Austin}. In fact, there are topological problems linked with proteins and enzymes, such as the topoisomerase. By investigating first the topological meaning of actions and reactions of proteins, and seeing them in  light of gesture theory and generalized interval theory, it is possible to envisage a musical description of these biological topics. This would also give a natural, philosophical meaning and foundation to some compositional strategies, highlighting the link between nature and music.

We can also imagine DNA as a small category, and music as another category; the sonification of DNA is made possible via a functor translating data and strings of bases into musical gestures. I will give later some detail about a possible way to realize such a project.

In previous work related to DNA, the technique implied different chords for combinations of proteins, and notes associated with protein sequences \cite{DNA_former2}.
Here, I want to see the double-helix structure resulting from drawing gestures, using analogies between sound and visuals. Fig. \ref{flow} shows the technique I followed to write a piece. This piece, titled {\em DNA}, has the following instrumentation: chimes, glockenspiel, celesta, harp, string quartet and double bass.
The bases A, C, T and G are rendered with different intervals (A as a minor third A-C, C as a major third C-E, T as a minor third B-D following the ``T'' name used in American solfeggio for B, and G as a major third G-B) and played by the percussionists. The connections between bases are represented by harp glissandi.

 \begin{figure}
 \centerline{
\includegraphics[width=10cm]{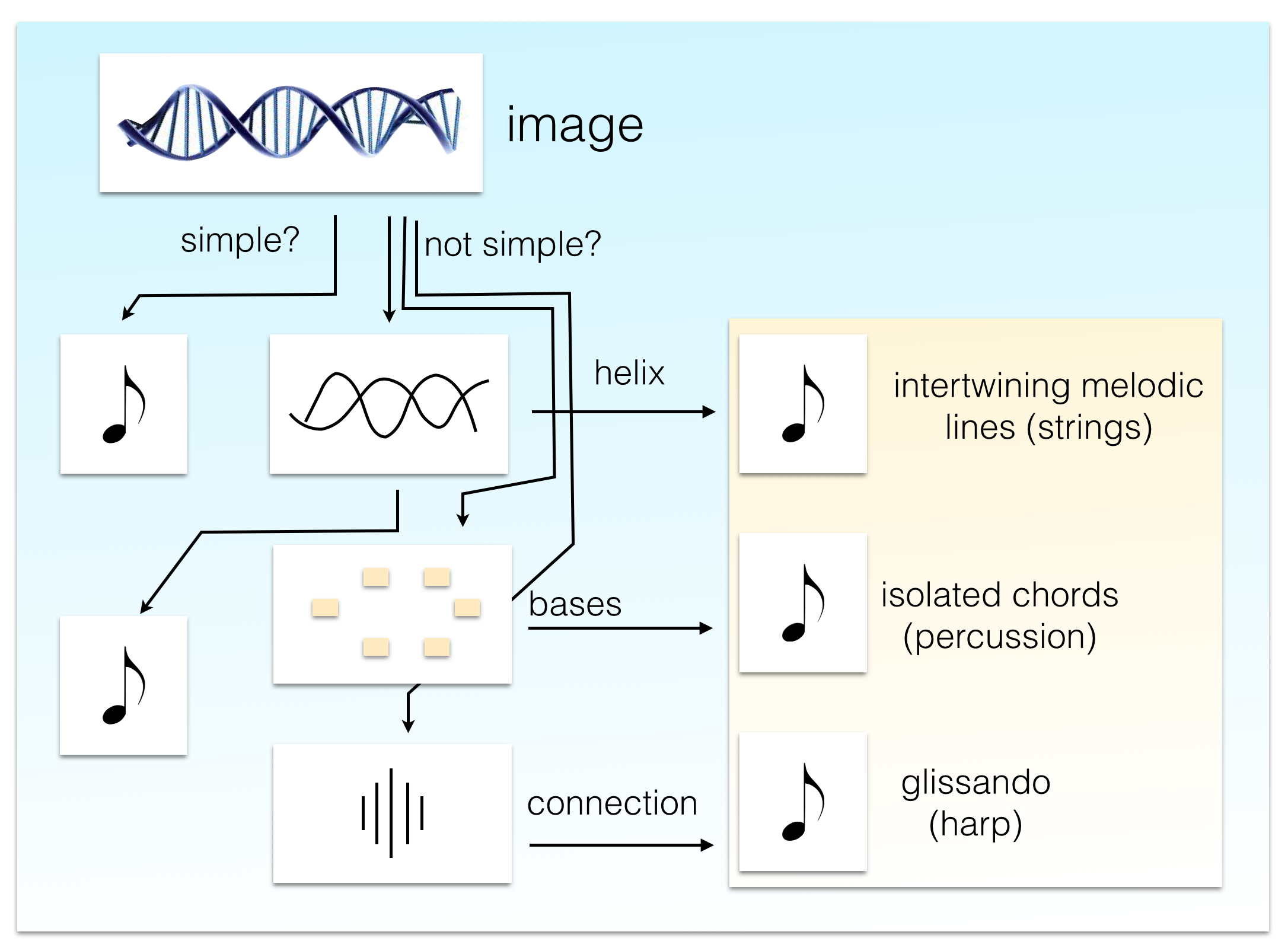}}
\caption{\label{flow} \footnotesize A flow diagram showing the technique behind the piece {\em DNA}. The piece contains the simultaneous musical rendition of three main elements the image was simplified into: the general shape of the double helix, the bases and their connection.}
\end{figure}

%A C\	 C A G	\ A G
%T G\ 	G T A \	T C
%
%
%La Do	\ Do La Sol	\	La Sol
%Si Sol	\ Sol Si La\		Si Do

How can one render the double helix musically?
I have two melodic lines, mutually exchanging and intertwining. The notes, given to strings, are found as sampling of sinusoids, then approximated, rescaled and transformed into musical notes, all via {\em Mathematica}$^{TM}$ software. The exchanging points are musical unisons.
The melodic-line structures are meant to create, in the mind of the listener, an intertwining structure, with some highlighted points (the bases).
Cases with 3-and 4-helix DNA are also taken into account later in the piece. %The original pair of melodies for the double helix is modified increasing duration values (augmentation) and transposed, while being played together with the variations,  to create an effect of counterpoint.
For this reason, the melodic lines representing the double helix were modified: their duration values were increased (augmentation) and their pitches transposed. The original and modified melodic lines are then played together to represent a 3- and 4-helix DNA, creating an effect of counterpoint.
Fig. \ref{DNA_sample} shows a fragment from the score.

A more detailed description would involve knotted protein \cite{bio2, bio3} and topoisomerase \cite{wang}.
The topological problems involved in the topoisomerase are supercoiling, knotting and concatenation.
A possible musical rendition of knotting (in DNA) has been described above. Concatenation can be musically rendered with superposed sequences, time reduction and deformation of the patterns to obtain more superposition. This may be compared with counterpoint transformations, voice leading and voice exchange procedures. 

Fig. \ref{coiling_1} shows a possible musical rendition of coiling. The repeated unison notes correspond to the superposition points in the diagram. One can compare this topic with studies about topological transformations of musical motives. One can also see the crossing as the result of a gestural operator on the musical pattern.

In chromosomes, there is supercoiling of coiled DNA structures.
With a slight abuse of notation, one can compare coiling and supercoiling with gesture and hypergesture.
I propose another musical rendition of coiling with a melody returning periodically on the same note, and, for supercoiling, one can think of several melodies with the same structure, that are repeated in repeated dynamic patterns: from {\em piano} to {\em mezzoforte} to {\em fortissimo} to {\em mezzoforte} to {\em piano} again, and so on. Thus, I use pitch to characterize coiling, and loudness to characterize supercoiling; see Fig. \ref{coiling_2}. Conceptually, it is a gesture in pitch, inside a larger gesture in loudness. One can compare all these steps, and the use itself of elementary gestures between two points, as an extension of Lewin's generalized intervals \cite{lewin}, in a tridimensional space environment. 

In summary, the topological meaning of action and reaction of proteins can be the object of musical composition via gestural concepts. This opens the way to a broader discussion on how nature can influence composition, and how artistic strategies give us a more direct way to understand nature.

 \begin{figure}
 \centerline{
\includegraphics[width=10cm]{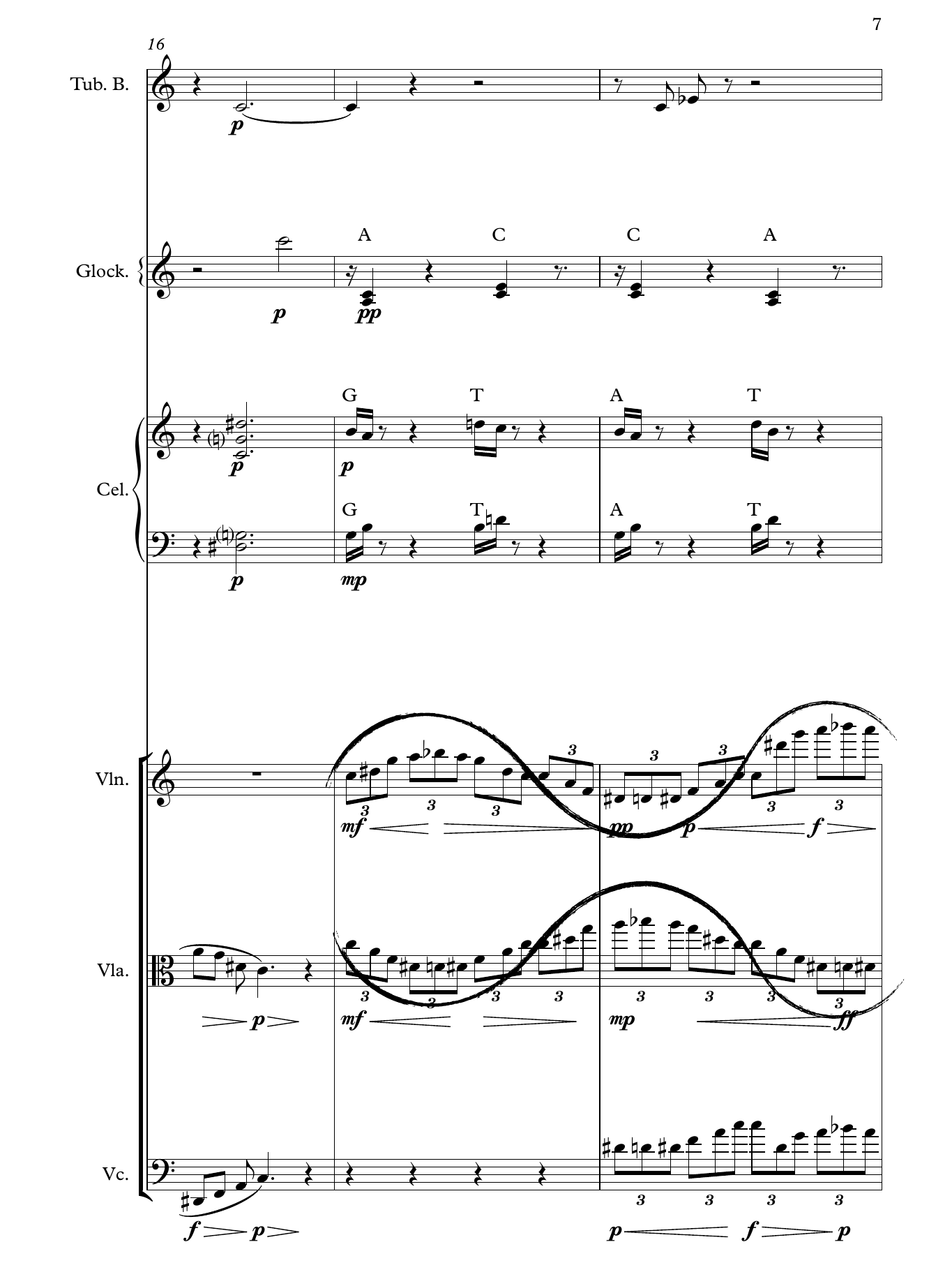}}
\caption{\label{DNA_sample} \footnotesize An extract from {\em DNA}. The curved line indicates the double helix, which then becomes a triple helix, and later a quadruple helix. The isolated intervals indicate the bases A, C, G and T of DNA, as described in the text.}
\end{figure}

 \begin{figure}
 \centerline{
\includegraphics[width=7cm]{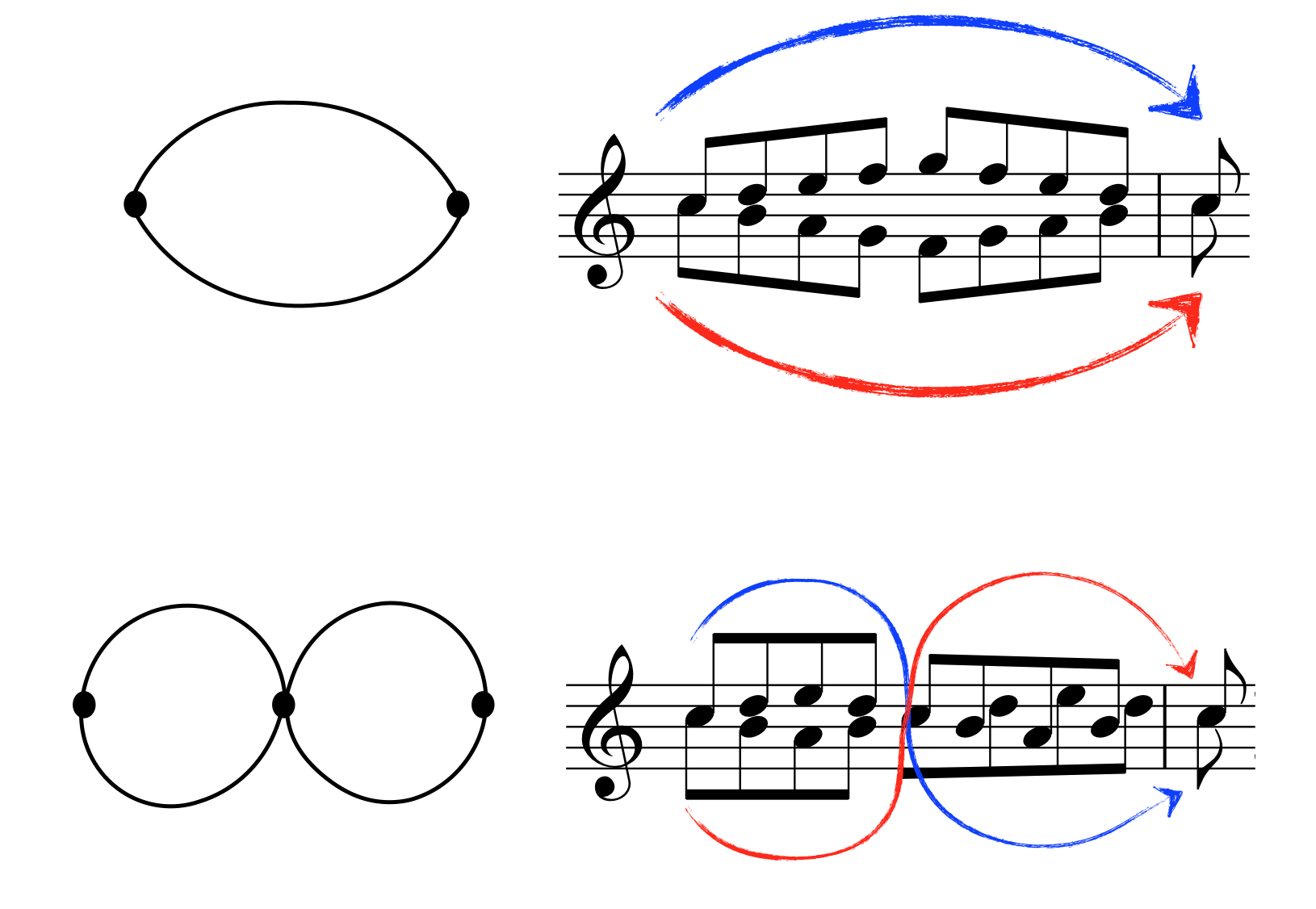}}
\caption{\label{coiling_1} \footnotesize A possibile musical rendition of coiling.}
\end{figure}

 \begin{figure}
 \centerline{
\includegraphics[width=12cm]{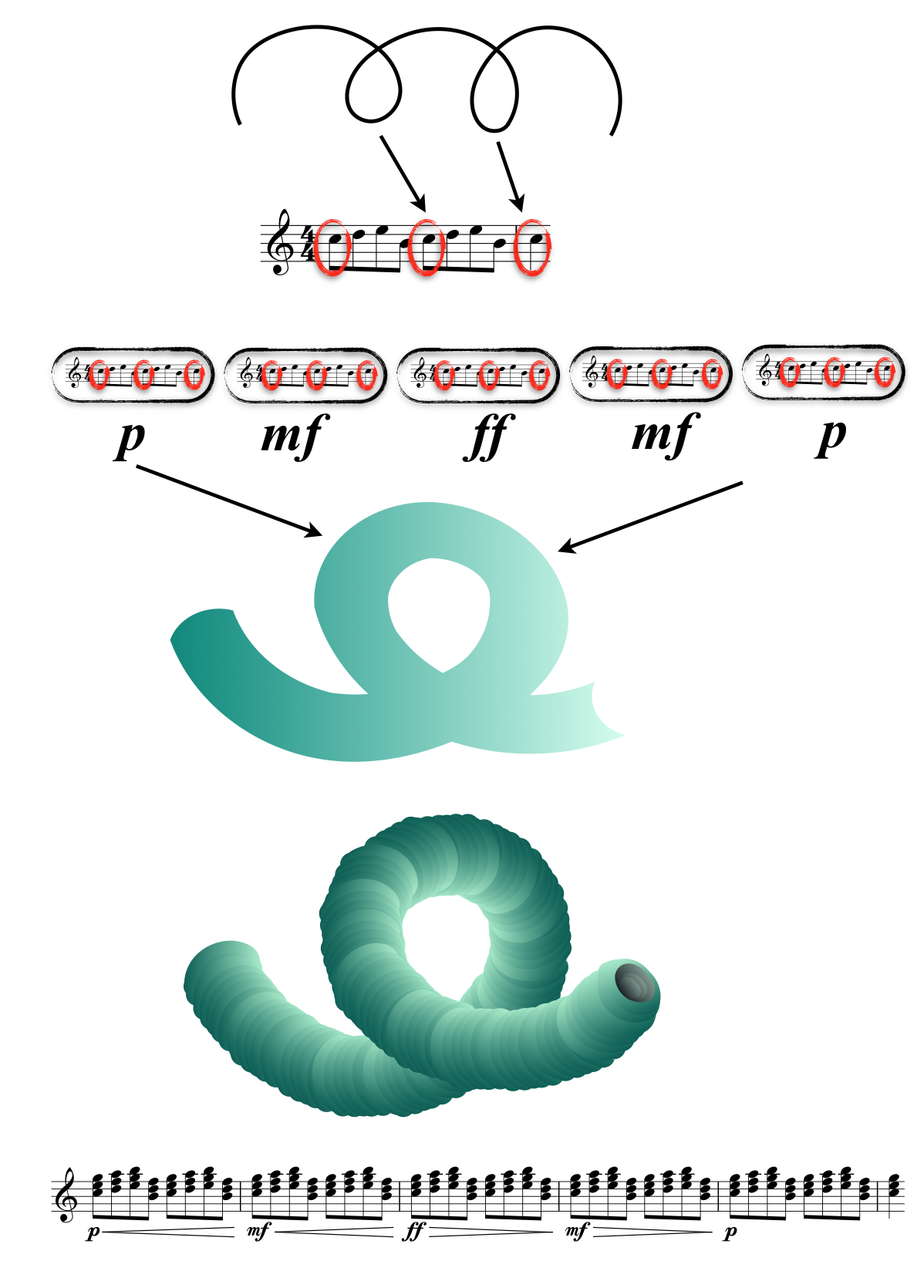}}
\caption{\label{coiling_2} \footnotesize A possibile musical rendition of supercoiling.}
\end{figure}

\section{Conclusion}\label{Conclusion}

In this article I started from recent work on the mathematical theory of musical gestures, taking into account new contributions and open questions, and trying to build up the basis for a more general and coherent model. The results of several previous studies are extended via knot theory and braided monoidal category theory. All the old and new elements are included in the definition of musical networks and their interaction with visual networks.

Due to the interdisciplinary nature of studies about knot theory and braided categories, I ended the text with a description of a new musical composition based on the DNA structure. The examples discussed stress the relationship between musical structures, mathematics and nature, and this can potentially open the way to new connections and mutual exchanges between art, mathematics and biology.

The main goal is not only about showing particular, specific case studies; I have also attempted here to extend mathematical thinking to other fields, and to analyse compositional strategies. In fact, I aim not only to develop a {\em descriptive} theoretical model but also a {\em prescriptive} model, in order to create new artworks. New analytical strategies also constitute a powerful tool to enhance musical creativity: once a method is envisaged, there are in principle endless applications and variations. Moreover, the identification of steps in flow diagrams allows composers to reduce the complexity of the construction of a score, breaking it down into more simple elements. For this reason, the approach to the composition of complex music may be easier, faster and more rationally based. The accomplished composer, as well as the student composer, may find within this theory helpful guidelines to build the structure of musical pieces, and creativity can be liberated within this flexible and creative structure.

Finally, further research might address the application of the discussed model to more analytical cases. Computer programs to analyse networks, quantifying the transformations at each step or branch, may be developed for this scope. Artificial intelligence developments may benefit from a mathematical and in particular categorical framework, as well as from artistic correlations supported by cognitive studies.
 
%\section{Acknowledgments}

%\section*{Acknowledgments}

%\newpage
 
%\maketitle

\bibliographystyle{apacite}
\bibliography{references}

\end{document}